\documentclass{amsart}

\usepackage{fullpage,enumerate,hyperref}

\usepackage[utf8]{inputenc}
\usepackage[T1]{fontenc}
\usepackage{amsmath,amsxtra,amsthm,amssymb,xr,enumerate}
\usepackage[all,cmtip]{xy}
\newtheorem{theorem}{Theorem}[section]
\newtheorem{lem}[theorem]{Lemma}

\newtheorem{corollary}[theorem]{Corollary}

\numberwithin{equation}{section}

\DeclareMathOperator{\ord}{ord}
\DeclareMathOperator{\Tr}{Tr}
\DeclareMathOperator{\Fil}{Fil}
\DeclareMathOperator{\Gal}{Gal}
\DeclareMathOperator{\Sel}{Sel}

\newcommand{\ZZ}{\mathbb{Z}}
\newcommand{\QQ}{\mathbb{Q}}
\newcommand{\FF}{\mathbb{F}}
\newcommand{\Zp}{\ZZ_p}
\newcommand{\Qp}{\QQ_p}
\newcommand{\Fp}{\FF_p}
\newcommand{\vp}{\varphi}
\newcommand{\Dcris}{\mathbb{D}_{\rm cris}(T_p(A))}

\begin{document}

\title{On the parity of supersingular Weil polynomials}

\begin{abstract}
Let $q$ be an odd power of a prime $p$ and let $A/\FF_q$ be a supersingular abelian variety of dimension $g$. We show that if $p>2g+1$, then the characteristic polynomial of the $q$-Frobenius is an even polynomial. This generalizes the well-known result on the vanishing of the trace of the $p$-Frobenius when $p>3$ for supersingular elliptic curves over $\FF_p$.
\end{abstract}

\author{David Ayotte}
\address{David Ayotte\newline
D\'epartement de Math\'ematiques et de Statistique\\
Universit\'e Laval, Pavillion Alexandre-Vachon\\
1045 Avenue de la M\'edecine\\
Qu\'ebec, QC\\
Canada G1V 0A6}
\email{david.ayotte.2@ulaval.ca}

\author{Antonio Lei}
\address{Antonio Lei\newline
D\'epartement de Math\'ematiques et de Statistique\\
Universit\'e Laval, Pavillion Alexandre-Vachon\\
1045 Avenue de la M\'edecine\\
Qu\'ebec, QC\\
Canada G1V 0A6}
\email{antonio.lei@mat.ulaval.ca}

\author{Jean-Christophe Rondy-Turcotte}
\address{Jean-Christophe Rondy-Turcotte\newline
D\'epartement de Math\'ematiques et de Statistique\\
Universit\'e Laval, Pavillion Alexandre-Vachon\\
1045 Avenue de la M\'edecine\\
Qu\'ebec, QC\\
Canada G1V 0A6}
\email{jean-christophe.rondy-turcotte.1@ulaval.ca}


\thanks{This research was carried out during the summer 2015, supported by an ISM summer scholarship (Ayotte) and a FGS bursary from Universit\'e Laval  (Rondy-Turcotte). The second-named authors' research is supported by the Discovery Grants Program 05710 of NSERC}
\maketitle

\section{Introduction} 

Let $E/\FF_p$ be an elliptic curve and let $T_\ell(E)$ be its Tate module at a  prime $\ell\ne p$. The characteristic polynomial of the $p$-Frobenius action  on $T_\ell(E)$ is of the form $P_E(X)=X^2-aX+p$, where $a$ is the trace of the Frobenius, which is equal to the integer $1+p-\# E(\Fp)$. In particular, it is independent of $\ell$. If furthermore $E$ is supersingular, then $p|a$. In this case, it is well-known that  $a$ is forced to be $0$ when $p>3$. This is a consequence of the Weil's bound which states that $|a|\le 2\sqrt{p}$. Indeed, assume that $a\ne0$. Then, the fact that $p|a$ implies that $|a|\ge p$. On combining this with the Weil's bound, we have $p\le 2\sqrt{p}$, which is equivalent to $p\le 4$.

We remark that the vanishing of $a$ has important consequences in Iwasawa Theory. For example, in \cite{kobayashi,iovitapollack}, the symmetry of the action of the Frobenius allows us to define plus and minus Selmer groups for  elliptic curves defined over $\QQ$ with supersingular reduction at a prime $p$. Consequently, we may formulate a main conjecture in terms of plus and minus $p$-adic $L$-functions defined in \cite{pollack}.

The goal of this article is to generalize the result "$p$ sufficiently large $\implies a=0$" to abelian varieties of higher dimensions. Let $q=p^n$ be a power of prime and let  $A/\FF_q$ be a $g$-dimensional supersingular abelian variety. Let $T_\ell(A)$ be its Tate module at a prime $\ell\ne p$. Let $P_A(X)$  be the characteristic polynomial of the $q$-Frobenius action  on $T_\ell(A)$. It is of the form
\[
P_A(X)=X^{2g}+a_1X^{2g-1}+\cdots +a_{g-1}X^{g+1}+a_gX^g+qa_{g-1}X^{g-1}+\cdots +q^{g-1}a_1X+q^g,
\]
where $a_1,\ldots,a_g\in\ZZ$. The first question we have to answer is, what is the right condition on $P_A(X)$  that would generalize the condition "$a=0$" for elliptic curves? On studying the list of supersingular Weil polynomials in \cite{singh}, we observe that when $n$ is odd and $g\le 7$, $P_A(X)$ is always an even polynomial (i.e. $a_1=a_3=\cdots =0$) for $p$ sufficiently large. This suggests that the right condition to replace "$a=0$" is $P_A(X)$ being an even polynomial. Indeed, on using a theorem of Manin-Oort which says that all roots of  $P_A(X)$ are of the form $\sqrt{q}\times\zeta$, where $\zeta$ is a root of unity, we may show that 
\begin{enumerate}
\item $|a_i|\le \binom{2g}{i} q^{i/2}$;
\item $\ord_p(a_i)\ge n\times i/2$.
\end{enumerate}
Here, (1) can be considered as generalized Weil's bounds, whereas (2) generalizes the condition "$p|a$" for elliptic curves. Suppose that $n$ is an odd integer. Combining (1) and (2), we may show as in the elliptic curve case that if $p>\binom{2g}{g}^2$ (when $g$ is odd) or $p>\binom{2g}{g-1}^2$ (when $g$ is even), then $a_i=0$ for all odd $i$. We give details of this proof in the appendix of this article.

Note that the bound $\binom{2g}{g}$ grows very rapidly as $g$ increases. For example, when $g=3$, $P_A(X)$ is even whenever $p>5$ (c.f. \cite[Theorem~12.1]{singh} as well as \cite{MN,xing}). However, $\binom{2g}{g}=20$ in this case. It is therefore desirable to find a smaller bound $C_g$ such that $p>C_g$ would imply that $P_A(X)$ is even for all supersingular abelian varieties over $\FF_{p^n}$  with $n$ being odd. In \S\ref{proof} below, we shall show the following theorem.
\begin{theorem}\label{thm:main}
Let $q=p^n$, where $p$ is a prime number and $n\ge1$ is an odd integer. Let $A/\FF_q$ be a supersingular abelian variety of dimension $g$. If $p>2g+1$, then $P_A(X)$ is an even polynomial.
\end{theorem}

In other words, we dramatically decrease the binomial bound to a linear bound.

We remark that our theorem does not hold without the condition that $n$ is odd. Indeed, if $q=p^{2m}$  for some $m\ge1$, then \cite[Theorem~12.2]{singh} tells us that $P_A(X)$ is not necessarily even for $p$ arbitrarily large. If we do not insist $A$ to be simple, we may take $P_A(X)=(X\pm p^m)^{2g}$ for example.

The structure of the article is as follows. We study an elementary result on cyclotomic polynomials in \S\ref{cyclo}. We then review an important result from \cite{singh} and give the details of the proof of Theorem~\ref{thm:main} in \S\ref{proof}. In \S\ref{iwasawa}, we briefly discuss some potential implications of our result for the Iwasawa theory of supersingular abelian varieties. Finally, in the appendix, we give a proof for a weaker version of Theorem~\ref{thm:main} involving binomial coefficients as discussed above.
\section*{Acknowledgement}
The second-named author would like to thank Kazim Buyukboduk for many interesting discussions on supersingular abelian varieties and related subjects. Thanks are also due to the anonymous referee for their constructive comments that led to significant improvements on the presentation of the article.
\section{The parity of cyclotomic polynomials} \label{cyclo}
Before we prove the main theorem of this section, we recall an elementary lemma on cyclotomic polynomials. 
\begin{lem}\label{lem:cyclo2}
  Let $p$ be a prime number and $n$, $k$ be positives integers. Then,
    \[
      \Phi_{p^kn}(X) = \begin{cases}
                         \Phi_n(X^{p^k}) &\text{ if } p \mid n;\\
                         \frac{\Phi_n(X^{p^k})}{\Phi_n(X^{p^{k-1}})} &\text{ if } p \nmid n.\\
                       \end{cases}
  \]
\end{lem}
This is a classical result. A proof can be found in \cite[Lemma~5]{Yimin} for example.
\begin{theorem}\label{thm:cyclo}
  Let $\Phi_n(X)$ be the $n^{th}$ cyclotomic polynomial. Then, $\Phi_n(X)$ is an even polynomial, i.e. $\Phi_n(X) = \Phi(-X)$, if and only if $n$ is divisible by $4$.
\end{theorem}
\begin{proof}
  First, suppose that $4 \mid n$. We may write $n = 2^km$ for $k \in \ZZ$, $k\geqslant2$ and $2 \nmid m$. From Lemma~\ref{lem:cyclo2},
  \begin{align*}
    \Phi_{2^km}(-X) &= \frac{\Phi_m((-X)^{2^k})}{\Phi_m((-X)^{2^{k-1}})} \\
                    &= \frac{\Phi_m(X^{2^k})}{\Phi_m(X^{2^{k-1}})} & \left(\text{because }k\geqslant2\right)\\
                    &= \Phi_{2^km}(X) \\
  \end{align*}
  and hence $\Phi_n(X) = \Phi_n(-X)$.
  
  Conversely, we show that if $4\nmid n$ then $\Phi_n(X)$ is not even. We consider two cases: (i) $2\nmid n$ and (ii) $2 \mid n$ but $4\nmid n$.
  \begin{enumerate}[(i)]
    \item Suppose that $2\nmid n$ and that $\Phi_n(X)$ is even. Let $\zeta$ be a root of $\Phi_n(X)$, then $-\zeta$ is a root too. Thus, there exists a root $\zeta'$ of $\Phi_n(X)$ such that $\zeta' = -\zeta$. Then,
  \begin{align*}
    (\zeta')^n &= (-\zeta)^n\\
            1 &= -1 &(\text{as $n$ is an odd integer}),
  \end{align*}
  which is obviously a contradiction. Hence $\Phi_n(X)$ cannot be even.
  
  \item Suppose that $n=2m$, where $m$ is an odd integer. Recall that $\Phi_{2m}(X) = \Phi_m(-X)$  (a proof  can be found in various places, for example \cite[Chapter 8]{Garrett} or \cite[Lemma 1.3]{ThR}). Hence,  by case (i),  $\Phi_n(X)$ is not even.
  \end{enumerate}  
\end{proof}

\section{Proof of Theorem~\ref{thm:main}}\label{proof}

Let $q=p^n$ be a positive odd power of a prime number $p$ and let $A/\FF_q$ be a $g$-dimensional supersingular abelian variety. Let $P_A(X)$ be the characteristic polynomial of the Frobenius endomorphism on the $\ell$-adic Tate module where $\ell$ is a prime distinct from $p$. We recall the following results regarding the roots of $P_A(X)$.

\begin{theorem}[Manin-Oort]\label{thm:maninoort}
All roots of $P_A(X)$ are of the form $\sqrt{q}\times\zeta$, where $\zeta$ is a root of unity.
\end{theorem}

As discussed in \cite[\S3, right after Remark~3.2]{singh}, if $\theta=\sqrt{q}\times\zeta$,  we may rewrite it as
$\sqrt{q^*}\times\zeta_{4t}$, where $q^*\in\{q,-q\}$ and $\zeta_{4t}$ is a $4t$-th primitive root of unity. Furthermore, the following theorem, which is Theorem~3.3 of \textit{op. cit.}, describes the minimal polynomial of such elements.

\begin{theorem}\label{thm:singh}
Let  $\theta=\sqrt{q^*}\times\zeta_{4t}$, where $q^*=\pm p^n$ for some positive odd integer $n$ and $\zeta_{4t}$ is a $4t$-th primitive root of unity. If either
\begin{enumerate}[(a)]
\item $q^*$ is odd and
\begin{enumerate}[(i)]
\item $t$ is even or
\item $p\nmid t$ or
\item $q^*\equiv 1\mod 4$ or
\end{enumerate}
\item $q^*$ is even and $t\not\equiv 2\mod4$,
\end{enumerate}
then the minimal polynomial  of $\theta$ over $\QQ$ is equal to 
\[
(\sqrt{q^*})^{\phi(4t)}\Phi_{4t}(X/\sqrt{q^*}),
\]
where $\phi$ is the Euler's totient function. 
\end{theorem}   
In \cite{singh}, when either (a) or (b) is satisfied, it is called the \textbf{\textit{full degree case}}. Otherwise, it is called the\textbf{ \textit{half degree case}}, for which the minimal polynomial of $\theta$ is also described explicitly in \textit{op. cit.}. But we shall not need this here.

\begin{lem}\label{lem:exceptional}
Suppose that $p>2g+1$, then the minimal polynomial of $\theta$ is an even polynomial.
\end{lem}
\begin{proof}
By Theorems~\ref{thm:cyclo} and \ref{thm:singh}, it is enough to show that the half degree case does not arise when $p>2g+1$.

Suppose that $\theta=\sqrt{q^*}\times\zeta_{4t}$ is a root of $P_A(X)$ such that the half degree case occurs. Write $\Psi(X)$ for the minimal polynomial of $\theta$ over $\QQ$. Then, $\Psi(X)|P_A(X)$.  In particular, 
\begin{equation}\label{eq:deg}
\deg(\Psi)\le\deg(P_A)=2g.
\end{equation}
 As $p>2g+1$, $p$ is an odd prime. Furthermore, since we are in the half degree case, we have $p|t$ and $2\nmid t$. This implies that 
\begin{align}
p-1|\phi(t);\label{eq:divide}\\
 \phi(4t)=2\phi(t),\label{eq:phi} 
\end{align}  on applying the formula $\phi(p_1^{n_1}\cdots p_r^{n_r})=p_1^{n_1-1}(p_1-1)\times\cdots \times p_r^{n_r-1}(p_r-1) $ for distinct primes $p_1,\ldots p_r$.

Recall from \cite[Remark~3.4]{singh} that $\deg(\Psi)=\frac{1}{2}\phi(4t)$. Therefore, on combining \eqref{eq:deg}, \eqref{eq:divide} and \eqref{eq:phi}, we deduce that
\[
p-1\le \phi(t)\le 2g.
\] 
Hence the result.
\end{proof}

This allows us to conclude the proof of Theorem~\ref{thm:main}. Since all roots of $P_A(X)$ are of the form $\sqrt{q}\times\zeta$ by Theorem~\ref{thm:maninoort}, $P_A(X)$ is a product of minimal polynomials of such elements. When $p>2g+1$, Lemma~\ref{lem:exceptional} says that all such minimal polynomials are even. Hence, $P_A(X)$ itself has to be even.

\section{Implications for Iwasawa theory}\label{iwasawa}

Let $E/\QQ$ be an elliptic curve with good  reduction at $p$. Let $K_\infty$ be the $\Zp$-cyclotomic extension of $\QQ$. We write $\Gamma=\Gal(K_\infty/K)$, $\Lambda=\Zp[[\Gamma]]$ and $K_n=K_\infty^{\Gamma^{p^n}}$ for $n\ge1$. Let $\Sel_p(E/K_\infty)$ denote the $p$-Selmer group of $E$ over $K_\infty$. It is well known in Iwasawa theory that the $\Lambda$-corank of $\Sel_p(E/K_\infty)$ is $0$ or $1$ depending on whether $E$ has ordinary or supersingular reduction at $p$.  In the former case, we may formulate a main conjecture relating the characteristic ideal of the dual Selmer group to a $p$-adic $L$-function. This conjecture is in fact now a theorem proved by Kato \cite{kato} and Skinner-Urban \cite{SU}.

 In the supersingular case, Kobayashi \cite{kobayashi} and Sprung \cite{sprung} have defined two new Selmer groups $\Sel_p^\pm(E/K_\infty)$, which we shall call plus and minus Selmer groups. These new Selmer groups turn out to have $\Lambda$-corank $0$. It is then possible to formulate a main conjecture as in the ordinary case, relating their characteristic ideals to some $p$-adic $L$-functions. One inclusion of these conjectures have been proved in \textit{op. cit.}. Recently, Wan and Sprung have announced the proof of the other inclusion of these conjectures in \cite{wan} and \cite{sprung2} respectively.

One common feature in the construction of the plus and minus Selmer groups in the works of Kobayashi and Sprung is a family of points on the elliptic curve $c_n\in E(K_{n,p}) $, where $K_{n,p}$ is the completion of $K_n$ at the unique prime above $p$. For all $n\ge 1$, we have the trace relation
\begin{equation}\label{eq:trace}
\Tr_{n+1/n}(c_{n+1})=a\times c_n-c_{n-1},
\end{equation}
where $a$ is the trace of the $p$-Frobenius on the Tate module $T_\ell(E)$ for some good prime  $\ell\ne p$.

When $a=0$, the simplicity of \eqref{eq:trace} allowed Kobayashi to give a very explicit description of his plus and minus Selmer groups. More precisely, we may define $\Sel_p^\pm(E/K_\infty)$ as the direct limit of $\Sel_p^\pm(E/K_n)$, which are defined by replacing the local condition at $p$ of the usual Selmer group  by $E^\pm(K_{n,p})\otimes\Qp/\Zp$,
where $E^\pm(K_{n,p})$ are given by
\begin{align*}
E^+(K_{n,p})&:=\{P\in E(K_{n,p}):\Tr_{n/m+1}P\in E(K_{m,p})\text{ for all even $m<n$} \};\\
E^-(K_{n,p})&:=\{P\in E(K_{n,p}):\Tr_{n/m+1}P\in E(K_{m,p})\text{ for all odd $m<n$} \}.
\end{align*}
Furthermore, it is shown in \cite[\S8]{kobayashi} that $E^\pm(K_{n,p})$ are generated by the Galois conjugates of the family of points $(c_m)_{m\le n}$ and that there is a short exact sequence relating these subgroups to the original elliptic curve, namely
\begin{equation}\label{eq:ses}
0\rightarrow E(\Qp)\rightarrow E^+(K_{n,p})\oplus E^-(K_{n,p})\rightarrow E(K_{n,p})\rightarrow 0.
\end{equation}
This was used in \textit{op. cit.} to prove a control theorem for the plus and minus Selmer groups.

When $a\ne0$, the Selmer groups defined in \cite{sprung} are much less explicit. So far, we do not know how to define an analogue of $E^\pm(K_{n,p})$ nor do we know whether a short exact sequence similar to \eqref{eq:ses} exists. In other words, the condition $a=0$ in Kobayashi's work plays an indispensable role. It is therefore natural to ask whether Theorem~\ref{thm:main} would allow us to better understand the supersingular Iwasawa theory of abelian varieties.

We recall that a multi-signed main conjecture for abelian varieties with supersingular reduction has been formulated in \cite{BL0}, generalizing the results in \cite{kobayashi,sprung}. Furthermore, this conjecture has been proved in \cite{BL1} under the hypothesis that certain Rubin-Stark elements from \cite{rubin1,rubin2} exist. The main conjecture in \cite{BL0} relies on the existence of a logarithmic matrix that is used to decompose Perrin-Riou's (conjectural) $p$-adic $L$-function from \cite{PR95} and to define the appropriate signed Selmer groups. It is therefore useful to have an explicit description of this logarithmic matrix.

The aforementioned logarithmic matrix is defined with respect to a chosen basis for the Dieudonné module $\Dcris$ of the $p$-adic Tate module of $A$. Let $v_1,\ldots,v_{2g}$ be a basis of $\Dcris$ that respects the filtration, that is, $v_1,\ldots,v_g$ are contained in $\Fil^0\Dcris$. Suppose $\vp$ is the Frobenius action on $\Dcris\otimes\Qp$, then its matrix with respect to such a basis is of the form
\begin{equation}
C_\vp=C\times\left(
\begin{array}{c|c}
I_{g}&0\\ \hline
0&\frac{1}{p}I_{g}
\end{array}
\right),\label{eq:matrix}
\end{equation}
where $C\in {\rm GL}_{2g}(\Zp)$ and $I_g$ is the $g\times g$ identity matrix. The logarithmic matrix is defined to be
\[
\lim_{n\rightarrow\infty}C_1\cdots C_n\cdot C_{\vp}^{-{n+1}},
\]
where $C_i$ for $i=1,\ldots, n$ is the matrix given by
\[
\left(
\begin{array}{c|c}
I_{g}&0\\ \hline
0&\Phi_{p^i}(X)I_{g}
\end{array}
\right)\times C^{-1}.
\]

The characteristic polynomial of $\vp$ is $X^{2g}P_A(X^{-1})$. In particular, Theorem~\ref{thm:main} tells us that if $p>2g+1$, then this is an even polynomial. We hope that this would allow us to give a more explicit description of the logarithmic matrix on choosing an appropriate basis for the Dieudonn\'e module. Indeed, when $g=1$, we may choose our basis so that $C=\begin{pmatrix}
0&-1\\
1&0
\end{pmatrix}$. Consequently, the logarithmic matrix is given by Pollack's plus and minus logarithm in \cite{pollack}. Furthermore, as described in \cite{ota}, it is possible to define a family of points $c_n\in A(K_{n,p})$ that satisfy a trace relation described by the coefficients of $P_A(X)$, similar to \eqref{eq:trace}. Therefore, when $P_A(X)$ is an even polynomial, this trace relation greatly simplifies. In future work, we shall explore the possibility of describing the signed Selmer groups of \cite{BL0} explicitly in this setting,  which would generalize the results of Kobayashi for elliptic curves that we described above.

\appendix
\section{Bounds involving binomial coefficients}

Let
\[
P_A(X)=X^{2g}+a_1X^{2g-1}+\cdots +a_{g-1}X^{g+1}+a_gX^g+qa_{g-1}X^{g-1}+\cdots +q^{g-1}a_1X+q^g
\]
be as in the main part of the article, where $q=p^n$ for some positive odd integer $n$. The goal of this appendix is to prove the following lemma.

\begin{lem}\label{lem:binom}
Let $k\in\{1,\ldots,g\}$ be an odd integer. If $a_k\ne 0$, then $p\le\binom{2g}{k}^2$.
\end{lem}

Consequently, we obtain a weaker version of Theorem~\ref{thm:main}.

\begin{corollary}
If $p> \binom{2g}{g}^2$ (when $g$ is odd) or  $p>\binom{2g}{g-1}^2$ (when $g$ is even), then $P_A(X)$ is an even polynomial.
\end{corollary}
\begin{proof}
This is because $\binom{2g}{k}\le\binom{2g}{r}$ for all $k\le r\le g$.
\end{proof}

We now prove Lemma~\ref{lem:binom}. Recall from Theorem~\ref{thm:maninoort} that the roots of $P_A(X)$ are of the form $\sqrt{q}\times\zeta_i$ for $i=1,\ldots,2g$, where $\zeta_i$ are some roots of unity. Vieta's formulas gives a link between $a_{k}$ and the roots of $P_{A}$. Namely,
\[
a_{k}=(-1)^{k}\sum_{1\le i_{1}< i_{2}< ...< i_{k} \le 2g}\left(\prod_{1\leq j\leq k}\zeta_{i_{j}}\right)p^{\frac{kn}{2}}.
\]
Since $a_k\in\ZZ$, we have $\ord_{p}(a_{k})\geq \left\lceil\frac{kn}{2}\right\rceil$. As we assume that $a_k\ne0$ and $kn$ is odd, this implies that
\[
|a_k|\ge p^{\frac{kn+1}{2}}.
\]
On the other hand, by the triangular inequality we can see that 
\begin{align*}
|a_{k}|&\leq\sum_{1\le i_{1}<i_{2}< ...< i_{k} \le 2g}\left(\prod_{1\leq j\leq k}|\zeta_{i_{j}}|\right){p}^{\frac{kn}{2}} \\
&=\binom{2g}{k}{p}^{\frac{nk}{2}}.
\end{align*}
On combining these two inequalities together, we deduce that
\[
\binom{2g}{k}p^{\frac{kn}{2}} \geq |a_{k}| \geq p^{\frac{kn+1}{2}} 
\]
and hence $p\leq \binom{2g}{k}^{2}$ as required.

\bibliographystyle{amsalpha}
\bibliography{references}

\providecommand{\bysame}{\leavevmode\hbox to3em{\hrulefill}\thinspace}
\providecommand{\MR}{\relax\ifhmode\unskip\space\fi MR }
\providecommand{\MRhref}[2]{%
  \href{http://www.ams.org/mathscinet-getitem?mr=#1}{#2}
}
\providecommand{\href}[2]{#2}
\begin{thebibliography}{Wan14}

\bibitem[BL15a]{BL1}
K\^az{\i}m B{\"u}y\"ukboduk and Antonio Lei, \emph{{C}oleman-adapted
  {R}ubin-{S}tark {K}olyvagin systems and supersingular iwasawa theory of {CM}
  abelian varieties.}, Proc. London Math. Soc. \textbf{111} (2015), no.~6,
  1338--1378.

\bibitem[BL15b]{BL0}
\bysame, \emph{Integral {I}wasawa theory of {G}alois representations for
  non-ordinary primes}, 2015, preprint, arXiv:1511.06986.

\bibitem[Gar08]{Garrett}
Paul~B. Garrett, \emph{Abstract algebra}, Chapman \& Hall/CRC, Boca Raton, FL,
  2008.

\bibitem[Ge]{Yimin}
Yimin Ge, \emph{Elementary properties of cyclotomic polynomials},
  \url{http://www.yimin-ge.com/doc/cyclotomic_polynomials.pdf}.

\bibitem[IP06]{iovitapollack}
Adrian Iovita and Robert Pollack, \emph{Iwasawa theory of elliptic curves at
  supersingular primes over {$\ZZ_p$}-extensions of number fields}, J. Reine
  Angew. Math. \textbf{598} (2006), 71--103.

\bibitem[Kat04]{kato}
Kazuya Kato, \emph{{$p$}-adic {H}odge theory and values of zeta functions of
  modular forms}, Ast\'erisque (2004), no.~295, ix, 117--290, Cohomologies
  $p$-adiques et applications arithm{\'e}tiques. III.

\bibitem[Kob03]{kobayashi}
Shin-ichi Kobayashi, \emph{Iwasawa theory for elliptic curves at supersingular
  primes}, Invent. Math. \textbf{152} (2003), no.~1, 1--36.

\bibitem[MN02]{MN}
Daniel Maisner and Enric Nart, \emph{Abelian surfaces over finite fields as
  {J}acobians}, Experiment. Math. \textbf{11} (2002), no.~3, 321--337, With an
  appendix by Everett W. Howe.

\bibitem[Ota14]{ota}
Kazuto Ota, \emph{A generalization of the theory of {C}oleman power series},
  Tohoku Math. J. (2) \textbf{66} (2014), no.~3, 309--320.

\bibitem[Pol03]{pollack}
Robert Pollack, \emph{On the {$p$}-adic {$L$}-function of a modular form at a
  supersingular prime}, Duke Math. J. \textbf{118} (2003), no.~3, 523--558.

\bibitem[PR95]{PR95}
Bernadette Perrin-Riou, \emph{Fonctions {$L$} {$p$}-adiques des
  repr\'esentations {$p$}-adiques}, Ast\'erisque (1995), no.~229, 198.

\bibitem[Rub92]{rubin1}
Karl Rubin, \emph{Stark units and {K}olyvagin's ``{E}uler systems''}, J. Reine
  Angew. Math. \textbf{425} (1992), 141--154.

\bibitem[Rub96]{rubin2}
\bysame, \emph{A {S}tark conjecture ``over {$\bold Z$}'' for abelian
  {$L$}-functions with multiple zeros}, Ann. Inst. Fourier (Grenoble)
  \textbf{46} (1996), no.~1, 33--62.

\bibitem[SMZ14]{singh}
Vijaykumar Singh, Gary McGuire, and Alexey Zaytsev, \emph{Classification of
  characteristic polynomials of simple supersingular abelian varieties over
  finite fields}, Funct. Approx. Comment. Math. \textbf{51} (2014), no.~2,
  415--436.

\bibitem[Spr12]{sprung}
Florian E.~Ito Sprung, \emph{Iwasawa theory for elliptic curves at
  supersingular primes: a pair of main conjectures}, J. Number Theory
  \textbf{132} (2012), no.~7, 1483--1506.

\bibitem[Spr15]{sprung2}
\bysame, \emph{The {I}wasawa main conjecture for elliptic curves at odd
  supersingular primes}, preprint, 2015.

\bibitem[SU14]{SU}
Christopher Skinner and Eric Urban, \emph{The {I}wasawa main conjectures for
  {$\rm GL_2$}}, Invent. Math. \textbf{195} (2014), no.~1, 1--277.

\bibitem[Tha00]{ThR}
R.~Thangadurai, \emph{On the coefficients of cyclotomic polynomials},
  Cyclotomic fields and related topics ({P}une, 1999), Bhaskaracharya
  Pratishthana, Pune, 2000, pp.~311--322.

\bibitem[Wan14]{wan}
Xin Wan, \emph{Iwasawa main conjecture for supersingular elliptic curves},
  preprint, arXiv:1411.6352, 2014.

\bibitem[Xin96]{xing}
Chaoping Xing, \emph{On supersingular abelian varieties of dimension two over
  finite fields}, Finite Fields Appl. \textbf{2} (1996), no.~4, 407--421.

\end{thebibliography}

\end{document}